\documentclass[11pt,letter]{amsart}
\usepackage[utf8]{inputenc}
\usepackage{graphicx}
\usepackage{amssymb}
\usepackage{amsmath}
\usepackage{epstopdf}
\usepackage{amsthm}

\usepackage{enumerate}

\usepackage{fourier}

\newtheorem{definition}{Definition} 
\newtheorem{proposition}[definition]{Proposition} 
\newtheorem{theorem}[definition]{Theorem} 
\newtheorem{lemma}[definition]{Lemma} 
\newtheorem{corollary}[definition]{Corollary}

\newcommand{\RR}{\mathbb{R}}

\newcommand{\ZZ}{\mathbb{Z}}

\renewcommand{\dim}{\mathsf{dim}\ }

\newcommand{\conv}{\operatorname{conv}}
\newcommand{\cone}{\operatorname{cone}}
\newcommand{\lev}{\operatorname{lev}}
\newcommand{\Lev}{\operatorname{Lev}}
\renewcommand{\deg}{\operatorname{deg}}
\renewcommand{\dim}{\operatorname{dim}}
\renewcommand{\min}{\operatorname{min}}

\newcommand{\fract}{\operatorname{frac}}
\newcommand{\integ}{\operatorname{int}}
\newcommand{\relint}{\operatorname{relint}}

\def\D{\Delta}
\def\D'{\Delta}

\title{Ehrhart $f^*$-coefficients of polytopal complexes are
non-negative integers}
\author{Felix Breuer}
\address{Department of Mathematics, San Francisco State University, 1600 Holloway Ave, San Francisco, CA 94132, USA}
\thanks{The author was supported by the DFG (Deutsche Forschungsgemeinschaft) grant BR 4251/1-1.}
\begin{document}
\maketitle

\begin{abstract}
The Ehrhart polynomial $L_P$ of an integral polytope $P$ counts the
number of integer points in integral dilates of $P$. Ehrhart
polynomials of polytopes are often described in terms of their Ehrhart $h^*$-vector
(aka Ehrhart $\delta$-vector), which is the vector of coefficients of
$L_P$ with respect to a certain binomial basis and which coincides
with the $h$-vector of a regular unimodular triangulation of $P$ (if
one exists). One important result by Stanley about $h^*$-vectors of polytopes is that their entries are always non-negative. However, recent combinatorial applications of Ehrhart theory give rise to polytopal complexes with $h^*$-vectors that have negative entries.

In this article we introduce the Ehrhart $f^*$-vector of polytopes or, more generally, of polytopal complexes $K$. These are again coefficient vectors of $L_K$ with respect to a certain binomial basis of the space of polynomials and they have the property that the $f^*$-vector of a unimodular simplicial complex coincides with its $f$-vector. The main result of this article is a counting interpretation for the $f^*$-coefficients which implies that $f^*$-coefficients of integral polytopal complexes are always non-negative integers. This holds even if the polytopal complex does not have a unimodular triangulation and if its $h^*$-vector does have negative entries. Our main technical tool is a new partition of the set of lattice points in a simplicial cone into discrete cones. Further results include a complete characterization of Ehrhart polynomials of integral partial polytopal complexes and a non-negativity theorem for the $f^*$-vectors of rational polytopal complexes. 
\end{abstract}


\section{Introduction}

For any set $X\subset\RR^n$ the \emph{Ehrhart function}
$L_X(k)=\ZZ^n\cap kX$ counts the number of lattice points in the
$k$-th dilate of $X$ for $1\leq k\in\ZZ$. Ehrhart's theorem states that if $P$ is a
lattice polytope then $L_P(k)$ is a polynomial in $k$ and, by
induction, the same holds for polytopal complexes with integral
vertices. \cite{Beck2007,Ehrhart1962,EugeneEhrhart1977}

Recently, a number of articles have appeared that realize various
combinatorial counting polynomials as Ehrhart functions of suitable
polytopal complexes and then apply results from Ehrhart theory to
prove theorems about these combinatorial functions.
\cite{Beck2006a,Beck2006,Breuer2009,Lisonek2007} In particular,
it is possible to obtain bounds on the coefficients of these
polynomials in this way. \cite{Breuer} For this purpose, the coefficients with respect to the monomial basis are not always easiest to work with. There are other bases of polynomial space that give rise to coefficient vectors such as the $h^*$- and $f^*$-vectors that are more amenable to analysis. These are defined as follows.

Let $p(k)$ be a polynomial in $k$ of degree at most $d$. Then there exist coefficients $h^*_i$ and $f^*_i$ for $i=0,\ldots,d$ such that 
\begin{eqnarray}
p(k) = \sum_{i=0}^d h^*_i {k+d-i \choose d} = \sum_{i=0}^d f^*_i {k-1 \choose i}. 
\end{eqnarray}
The coefficients $h^*_i$ and $f^*_i$ depend both on $p$ and on $d$, so
we will sometimes write $h^*_i(p,d)$ and $f^*_i(p,d)$ to make this
dependency explicit. The vectors $(h^*_0,\ldots,h^*_d)$ and $(f^*_0,\ldots,f^*_d)$ are
called the \emph{$h^*$- and $f^*$-vectors} of $p$ and their entries
are the \emph{$h^*$- and $f^*$-coefficients} of $p$, respectively.
Note that the $h^*$-vector also goes by the name of Ehrhart
$\delta$-vector. \cite{Stapledon2009} Whenever we refer to the $h^*$-
or $f^*$-vector of an integral polytope or polytopal complex $P$, we mean the $h^*$-
or $f^*$-vector of its Ehrhart polynomial $L_P$. For more details on these vectors and, most importantly, the motivation for defining them we refer the reader to Section~\ref{sec:prelim-f-and-h}.

One famous result about $h^*$-vectors is Stanley's theorem which asserts that the $h^*$-coefficients of the Ehrhart
polynomial of an integral polytope are always non-negative integers.
\cite{Stanley1980} Behind this theorem lies a beautiful interpretation, due to Ehrhart, of the $h^*$-coefficients of the Ehrhart polynomial of a simplex $\Delta$ as counting lattice points at various heights in the fundamental parallelepiped of the cone over the homogenization of $\Delta$. \cite{Ehrhart1962,EugeneEhrhart1977}

While $h^*$-vectors of integral polytopes are always non-negative, $h^*$-vectors of integral polytopal complexes may well have negative entries. Moreover, polytopal complexes with negative $h^*$-coefficients appear in natural combinatorial applications. For example, coloring complexes of uniform hypergraphs can have negative $h^*$-coefficients. Their $f^*$-vector, however, is always non-negative. See Section~\ref{sec:prelim-col-comp} and \cite{Breuer2000} for details.

Thus, the question arises whether this is always true: Do polytopal
complexes always have a non-negative integral $f^*$-vector? The
purpose of this article is to give a positive answer to this question.

Our main result is a counting interpretation of the
$f^*$-vector of a simplex $\Delta$, in the spirit of the classic counting
interpretation of the $h^*$-vector of a simplex. Given a relatively
open lattice simplex $\Delta$, the $f^*$-vector counts the number of
so-called atomic lattice points at different heights in the
fundamental simplex of the cone over the homogenization of $\Delta$.
More precisely:

\begin{theorem}
\label{thm:counting-interpretation}
Let $\Delta\subset\mathbb{R}^n$ be an open lattice
simplex, let $d'\geq d=\text{dim}(\Delta)$ and let
$f^*(L_\Delta,d')=(f^*_0,\ldots,f^*_{d'})$. Then $f^*_i$ counts the number of
atomic lattice points in the half-open fundamental simplex of
$\cone_{\mathbb{R}}(\Delta\times\{1\})$ at level $i+1$.
\end{theorem}

The definitions of the fundamental simplex, atomic lattice points and
their level are given in Section~\ref{sec:partition}. An open lattice
simplex is the relative interior of a lattice simplex.

From this counting interpretation we can immediately obtain a complete
characterization of the $f^*$-vectors of
integral partial polytopal complexes. Here, an integral partial
polytopal complex is any set that can be written as the disjoint union
of relatively open polytopes with integral vertices. 

\begin{theorem}
\label{thm:characterization}
A vector is the $f^*$-vector of some integral
partial polytopal complex if and only if it is integral and
non-negative.
\end{theorem}

In particular, this gives us the desired non-negativity result for
$f^*$-vectors of polytopal complexes.

\begin{theorem} 
\label{thm:non-negative}
Every integral polytopal complex, and in particular every lattice polytope, has a non-negative integral $f^*$-vector.
\end{theorem}

The crucial point here is that the $f^*$-vector is non-negative and
integral even if the complex does not have unimodular triangulation
and even if its $h^*$-vector has negative entries. Note that
non-negativity of the $f^*$-vector follows automatically if the complex has a unimodular triangulation or if the
$h^*$-vector is non-negative. This means that
Theorem~\ref{thm:non-negative} gives a new result only if the complex
in question is non-convex and does not have a non-negative $h^*$-vector.
But, as we already mentioned, there are non-convex polytopal complexes with
negative $h^*$-coefficients that do appear in practical applications.

The key technical ingredient that goes into the above counting
interpretation is the following partition of the set of lattice points
in a simplical cone into ``discrete cones''. 

\begin{theorem}
\label{thm:partition}
Let $v_1,\ldots,v_d$ be linearly independent integer
vectors in $\mathbb{Z}^n$ for $n\geq d$. Then
\begin{eqnarray}
\label{eqn:partition}
 \relint(\cone_{\mathbb{R}}(v_1,\ldots,v_d)) \cap\mathbb{Z}^n = \bigcup_{z\text{ atomic}} z +\cone_{\mathbb{Z}}(v_1,\ldots,v_{\lev(z)}),
\end{eqnarray}
where the union ranges over all atomic lattice points in the half-open
fundamental simplex of
$\cone_{\mathbb{R}}(v_1,\ldots,v_d)$ and this union is disjoint.
\end{theorem}

Here $\lev(z)$ denote the level of $z$ and $\cone_{\mathbb{R}}(v_1,\ldots,v_d)$ refers to all
non-negative linear combinations of the $v_i$ whereas
$\cone_{\mathbb{Z}}(v_1,\ldots,v_d)$ refers to all non-negative,
integral linear combinations of the $v_i$. Again, we refer to
Sections~\ref{sec:preliminaries} and \ref{sec:partition} for details.

Theorem~\ref{thm:partition} is much more general then necessary for
Theorems~\ref{thm:counting-interpretation}, \ref{thm:characterization}
and \ref{thm:non-negative} and is the main technical result of this
article. In particular, Theorem~\ref{thm:partition} can be used to
obtain a counting interpretation and a non-negativity theorem in the rational case.

\begin{theorem}
\label{thm:fstar-count-rational}
Let $\Delta\subset\mathbb{R}^n$ be an open rational simplex, let $d'\geq d=\text{dim}(\Delta)$ and $m$ be a positive integer such that $m\Delta$ is integral. There exist polynomials $p_0,\ldots,p_{m-1}$ such that for all integers $k\geq0$ and $0 \leq l < m$ with $(k,l)\not=(0,0)$ the Ehrhart function of $\Delta$ satisfies $L_\Delta(km+l)=p_l(k)$. Then for all $0\leq i\leq d$ and all $0\leq l< m$ the $f^*$-coefficient $f^*_i(p_{l},d)$ counts the number of atomic lattice points $z$ in the half-open fundamental simplex of $\cone_{\mathbb{R}}(\Delta\times\{m\})$ at level $i+1$ with $z_{n+1}=im+l+ 1$.
\end{theorem}

This counting interpretation yields a non-negativity theorem for the
$f^*$-vector just as in the integral case. The $f^*$-vector of a
rational polytopal complex is the vector of all numbers
$f^*_i(p_l,d)$ for $i=0,\ldots,d$ and $0\leq l < m$, see Section~\ref{sec:rational-case} for details.

\begin{theorem}
\label{thm:non-negativity-rational}
Any rational partial polytopal complex has a non-negative integral $f^*$-vector.
\end{theorem}

Interestingly, there is another variant of
Theorem~\ref{thm:fstar-count-rational} that expresses the Ehrhart
function of a rational simplex in terms of restricted partition functions. For
our purposes the restricted partition function $p_{m_1,\ldots,m_d}(k)$ is given
by
\[ p_{m_1,\ldots,m_d}(k) = \# \left\{
(\lambda_1,\ldots,\lambda_d) \; \middle|\; 0\leq\lambda_i\in\ZZ, \sum_{i=1}^d
\lambda_i m_i = k \right\}, \]
see Section~\ref{sec:rational-case} for details. Then
Theorem~\ref{thm:partition} allows us to write the Ehrhart function of
a rational simplex in terms of restricted partition functions in the following way.

\begin{theorem}
\label{thm:fstar-count-rational-skew}
Let $\Delta\subset\mathbb{R}^n$ be an open lattice simplex with vertices $v_1,\ldots,v_{d+1}$ and let $m_1,\ldots,m_{d+1}$ be minimal positive integers such that $m_iv_i$ is integral for all $1\leq i \leq d+1$. Then 
\[
L_\Delta(k) = \sum_{i=0}^{d}\sum_{s=0}^{S} c_{i,s} \cdot p_{m_1,\ldots,m_{i+1}}(k-s)
\]
for all $1\leq k$ where $S=\sum_{i=1}^{d+1} m_i$ and $c_{i,s}$ denotes the number of atomic lattice points $z$ at level $i+1$ in the fundamental simplex of $\cone_{\RR}(a_1,\ldots,a_{d+1})$ with $z_{n+1}=s$. Here, $a_i=(m_iv_{i,1},\ldots,m_iv_{i,n},m_i)$ for all $i=1,\ldots,d+1$.
\end{theorem}

\bigskip

This paper is organized as follows. In Section~\ref{sec:preliminaries} we give some preliminary definitions, sketch a classic proof of the non-negativity of $h^*$-vectors for polytopes and give an example of a natural simplicial complex with a negative $h^*$-vector. In Section~\ref{sec:partition} we present the partition of the set of lattice points in an open simplicial cone into discrete subcones, which is the main technical result of this article. In Section~\ref{sec:fstar-counting} we use this partiton result to give a counting interpretation of the $f^*$-coefficients of a simplex, prove the non-negativity of the $f^*$-vector and give a complete characterization of the Ehrhart polynomials of integral partial polytopal complexes. Up to this point we have mainly worked with integral polytopes, to make the ideas behind the construction more transparent. However, most of our results apply to the rational case as well. In Section~\ref{sec:rational-case} we introduce $f^*$-vectors of rational polytopes, give a counting interpretation, prove the non-negativity of the $f^*$-vectors of rational partial polytopal complexes and relate Ehrhart functions of rational simplices to restricted partition functions.

\section{Preliminaries}
\label{sec:preliminaries}

\emph{Note: A comprehensive definition of all notions from polytope theory, Ehrhart theory or generating function theory that we make use of is out of scope of this article. For any undefined terms we refer the reader to \cite{Beck2007,Schrijver1986,Ziegler1997}.} 

\subsection{Geometry}

A \emph{polytope} is the convex hull of finitely many points. A
\emph{supporting hyperplane} of a polytope $P$ is a hyperplane such
that $P$ is contained in one of the two corresponding closed
half-spaces. A \emph{face} of $P$ is the intersection of a supporting
hyperplane with $P$. By convention $P$ is a face of itself as well.
The \emph{dimension} of $P$ is the dimension of its affine hull. The
faces of dimension 0 are called \emph{vertices}. 

A polytope is
\emph{integral} if all its vertices are elements of the \emph{integer
lattice} $\ZZ^n$, where $n$ is the dimension of the ambient space.
Integral polytopes are also called \emph{lattice polytopes}. Two
polytopes $P,Q$ are \emph{lattice equivalent} if there is an affine
isomorphism $\phi$ of the ambient space with $\phi(P)=Q$ that induces
a bijection on the integer lattice $\ZZ^n$.   

The \emph{relative interior} of a polytope $P$ is the interior of $P$ taken with respect
to its affine hull. We also use the term \emph{open polytope} to refer
to the relative interior of a polytope. When we speak of the faces of
an open polytope, we mean the faces of its closure. Every polytope is
the disjoint union of the relative interiors of its faces.

A \emph{simplex} is the convex hull of finitely many affinely
independent points. A simplex of dimension $d$ has exactly $d+1$
vertices. The \emph{standard simplex} $\Delta^d$ of dimension $d$ is the convex
hull of $d+1$ standard unit vectors. An integral simplex is
\emph{unimodular} if it is lattice equivalent to a standard simplex.

A \emph{polytopal complex} is a finite set of polytopes $K$ with the
following two properties: 1) If $P\in K$ and $Q$ is a face of $P$,
then $Q\in K$. 2) If $P,Q\in K$, then $P\cap Q\in K$ and $P\cap Q$ is
a face of both $P$ and $Q$. The elements of $K$ are also called
\emph{faces} of $K$. The \emph{dimension} of $K$ is the maximum
dimension of any face of $K$. The \emph{support} of $K$ is the union of
all polytopes in $K$. A polytopal complex is
\emph{integral} if all of its faces are integral.

A \emph{simplicial complex} is a polytopal complex whose faces are
simplices. A \emph{triangulation} of a set $X\subset\RR^n$ is a
simplicial complex whose support is $X$. A simplical complex is
\emph{unimodular} if all of its faces are unimodular. Note that not
all integral polytopes, not even all integral simplices, have a
unimodular triangulation.

The \emph{$f$-vector} of a
$d$-dimensional simplicial complex $K$ is the vector
$(f_0,f_1,\ldots,f_d)$ where $f_i$ is the number of $i$-dimensional
faces of $K$. The \emph{$h$-vector} of $K$ is the vector
$(h_0,\ldots,h_{d+1})$ defined in terms of the
$f$-vector via
\[
h_k  = \sum_{i=0}^k(-1)^{k-i}{d-i \choose d-k} f_{i-1}
\]
for $k=0,\ldots,d+1$ where $f_{-1}:=1$. Note that the $h$-vector has one more entry than the $f$-vector but $h_0=1$ is fixed.

\subsection{Ehrhart theory}

As mentioned in the introduction, our point of departure is Ehrhart's
theorem, which states that for any integral polytope $P\in\mathbb{R}^d$
there exists a polynomial $L_P(k)$ such that
\[\# \left( \mathbb{Z}^d\cap kP \right) = L_P(k)\] for all $1 \leq k\in\mathbb{Z}$. 

It is straightforward to see that Ehrhart's theorem carries over to polytopal complexes. However, many applications go one step further and work with ``partial'' polytopal complexes instead, where some faces are missing. In particular, inside-out polytopes are examples of half-open polytopal complexes that are widely used in combinatorial applications of Ehrhart theory. \cite{Beck2006a,Beck2006,Breuer} Let us now make precise what we mean by ``partial'' in this context. 

As defined in the previous section, a (relatively) open polytope is
the relative interior of a polytope. The vertices, faces and facets of
an open polytope are defined to be the vertices, faces and facets of
its closure. Note that thus the vertices of an open polytope are not
contained in the open polytope. An open polytope is called integral if all its vertices are integral. 

For any polytopal complex $K$, the support of $K$ is the disjoint
union of the relative interiors of all faces of $K$. This motivates the following definition: A \emph{partial polytopal complex} is a disjoint union of open polytopes. The difference between a polytopal complex and a partial polytopal complex is therefore simply that some of the relatively open faces of the polytopal complex (that would need to be included because a polytopal complex has to be closed under passing to faces) have been removed.

Two important special cases of partial polytopal complexes are the following. 

A ``half-open'' polytope is a set of the form $P\setminus
\bigcup_{i=1}^l \sigma_i$ where $P$ is a polytope and the $\sigma_i$
are faces of $P$. Every half-open polytope is the support of some
partial polytopal complex. The half-open simplices that we are going to meet in the next section are examples of this.

A \emph{relative simplicial complex} is a set of simplices of the form
$K\setminus K'$ where $K$ is a simplicial complex and $K'$ is a
subcomplex of $K$. Relative simplcial complexes can be written as
partial polytopal complexes. They appear, for example, in
Steingrímsson's construction of the coloring complex. \cite{Steingrimsson2001} Relative polytopal complexes can be defined similarly and again they can be realized as partial polytopal complexes. Inside-out polytopes are examples of relative polytopal complexes. \cite{Breuer2009d}

\subsection{$f^*$- and $h^*$-vectors}
\label{sec:prelim-f-and-h}

Let us denote by $\Delta^d_i$ the $d$-dimensional standard simplex with
$i$ open facets, i.e.,
\[\Delta^d_i := \left\{x\in\mathbb{R}^{d+1} \;\middle|\; \sum_{j=0}^dx_j = 1, x_j > 0 \text{ for } 0\leq j < i, x_j \geq 0 \text{ for } i \leq j \leq d \right\}.\]

It turns out that $L_{\Delta^d_i}(k)= {k+d-i \choose d}$ for
$i=0,\ldots,d+1$ and in particular
$L_{\Delta^d_{d+1}}(k)={k-1 \choose d}$ where
$\Delta^d_{d+1}=\relint{\Delta^d}$ is the relative interior of the
standard $d$-dimensional simplex. This has the following immediate
consequences for a $d$-dimensional integral polytopal complex $C$.

\begin{enumerate}
\item {If $C$ has a unimodular triangulation $K$, then $C$ can be written as
  a disjoint union of relatively open unimodular simplices
  $\Delta^i_{i+1}$ of varying dimension $i=0,\ldots,d$. Thus
  \[L_C(k)=\sum_{i=0}^d f^*_i { k-1 \choose i }\] where the coefficients
  $f^*_i$ count the number of $i$-dimensional relatively open unimodular
  simplices appearing in the disjoint union. In this case the $f$-vector
  of the simplicial complex $K$ coincides with the vector of
  coefficients $f^*_i$, which explains the name.}
\item If $C$ has a unimodular triangulation $K$ that can be written as a disjoint union of unimodular half-open
  simplices $\Delta^d_i$, $i=0,\ldots,d$ of fixed dimension $d$, then
  \[L_C(k)=\sum_{i=0}^d h^*_i { k+d-i \choose d }\] where the
  coefficients $h^*_i$ count the number of $i$-dimensional relatively
  open unimodular simplices appearing in the disjoint union. In particular, if $K$ is a shellable\footnote{See \cite{Ziegler1997} for the definition of shellable.} complex that is a topological ball then the $h$-vector of $K$ coincides with the
  vector of coefficients $h^*_i$, which explains the name. Note that if $K$ is not a topological ball, then $h_{d+1}$ is non-zero in general and the $h^*$- and $h$-vectors may differ.
\end{enumerate}

If $C$ does not have a unimodular triangulation, we can still define the
$h^*$- and $f^*$-vectors of $C$. In fact, we can define $h^*$- and
$f^*$-vectors for arbitrary polynomials by proceeding as sketched in the introduction.

For any integer $i$,
\[{k \choose i} = \frac{k\cdot(k-1)\cdot \ldots \cdot (k-i+1)}{i!}\] is
a polynomial in $k$ of degree $i$. Moreover, both
\[\left\{ {k-1 \choose i} \;\middle|\; i=0,\ldots,d \right\} \text{ and } \left\{ {k + d -i \choose d} \;\middle|\; i=0,\ldots,d \right\}\]
form bases of the vector space of polynomials in $k$ of degree at most
$d$. Thus, for any non-negative integer $d$ and any polynomial $p(k)$ of
degree at most $d$ we can define vectors $f^*(p,d)=(f^*_0,\ldots,f^*_d)$
and $h^*(p,d)=(h^*_0,\ldots,h^*_d)$ by
\begin{eqnarray*}p(k) &=& \sum_{i=0}^d f^*_i {k-1 \choose i} \\ p(k) &=& \sum_{i=0}^d h^*_i {k+d-1 \choose d}.\end{eqnarray*}

We call $f^*(p,d)$ the $f^*$-vector of $p$ and the numbers $f^*_i$ the
$f^*$-coefficients of $p$. Similarly, we call $h^*(p,d)$ the
$h^*$-vector of $p$ and the numbers $h^*_i$ the $h^*$-coefficients of
$p$.

At this point, it important to call attention to the following subtlety:
$h^*$ depends on the choice of $d$, whereas $f^*$ does not. More
precisely, the $f^*$-vector has the following property. Let $p$ be any
polynomial and let $d_1,d_2 \geq \deg(p)$ be any two
integers. Then $f^*_i(p,d_1)=f^*_i(p,d_2)$ for all
$0\leq i \leq \min(d_1,d_2)$. This statement is false for $h^*$.
Despite this difference, we are going to suppress $p$ and $d$ in our
notation for both $f^*$ and $h^*$ whenever it is clear from context
which $p$ and $d$ are meant.

Now that we have defined the $f^*$- and $h^*$-vectors of polynomials, we
can define $f^*$- and $h^*$-vectors of polytopes (and more generally
polytopal complexes) via the Ehrhart function.

Let $K$ denote a polytopal complex. Then the $f^*$- and $h^*$-vectors of
$K$ are, respectively, defined by
\begin{eqnarray*}
f^*(K,d) &=& f^*(L_K,d) \\
h^*(K,d) &=& h^*(L_K,d),
\end{eqnarray*}
where $d\geq \text{dim}(K)$.

If we do not specify $d$ explicitly, it is understood that
$d=\text{dim}(K)$, that is, $f^*(K)=f^*(K,\text{dim}(K))$ and
$h^*(K)=h^*(K,\text{dim}(K))$.

\subsection{Generating function point of view}

Classically, the $h^*$-vector is defined in terms of generating
functions.

\begin{proposition}[c.f.~Lemma~3.14 in \cite{Beck2007}]
 If $p$ is a polynomial of degree at
most $d$, then
\[\frac{h^*_0z^0+\ldots+h^*_dz^d}{(1-z)^{d+1}} = \sum_{k\geq 0} p(k) z^k.\]
\end{proposition}

A similar statement can be made about the $f^*$-vector.

\begin{proposition} 
If $p$ is a polynomial of degree at most $d$, then
\[\frac{f^*_0z^1}{(1-z)^1} + \cdots + \frac{f^*_dz^{d+1}}{(1-z)^{d+1}} = \sum_{k\geq0} p(k)z^k.\]
\end{proposition}

\begin{proof} 
The coefficient of $z^k$ in the Laurent expansion of
$\frac{1}{(1-z)^{j+1}}$ is precisely
$L_{\Delta^j}(k)={ k + j \choose j}$, the number of lattice points in
the $k$-th dilate of a $j$-dimensional unimodular simplex. Thus
\[ \frac{z^{j+1}}{(1-z)^{j+1}} = \sum_{k \geq 0} {k - 1\choose j} \;z^k\]
which yields the desired identity. 
\end{proof}

\begin{corollary} The $f^*$- and $h^*$-vectors of a polynomial $p$ satisfy
\[ h_0^*z^0+\ldots+ h^*_dz^d = \sum_{j=0}^d f_j^*z^{j+1}(1-z)^{d-j}. \]
\end{corollary}

\subsection{Counting interpretation for the $h^*$-vector}

Given linearly independent integer vectors
$a_1,\ldots,a_n\subset\mathbb{Z}^d$ we define the \emph{cone} over the
$a_i$ by
\[\cone_{\mathbb{R}}(a_1,\ldots,a_n) = \left\{ x\in\mathbb{R}^d \;\middle|\; x=\sum_{i=1}^n \lambda_i a_i, 0\leq \lambda_i\in\mathbb{R} \right\}.\]

Instead of allowing real coefficients $\lambda_i$, we can also restrict
ourselves to integral coefficients. In this way, we obtain the
\emph{discrete cone} over the $a_i$ which is
\[\cone_{\mathbb{Z}}(a_1,\ldots,a_n) = \left\{ x\in\mathbb{R}^d \;\middle|\; x=\sum_{i=1}^n \lambda_i a_i, 0\leq \lambda_i\in\mathbb{Z} \right\}.\]

The \emph{fundamental parallelepiped} $\Pi(a_1,\ldots,a_n)$ of the cone
is
\[\Pi(a_1,\ldots,a_n) = \left\{ x\in\mathbb{R}^d \;\middle|\; x=\sum_{i=1}^n \lambda_i a_i, 0\leq \lambda_i <0, \lambda_i\in\mathbb{R} \right\}.\]

The crucial property of the fundamental parallelepiped is that it tiles
the cone. That is, the cone can be written of as the disjoint union of
integral translates of the parallelepiped, where the translation vectors
are precisely the elements of the discrete cone. In terms of the
Minkowski sum, this can be written simply as:
\[\cone_{\mathbb{R}} (a_1,\ldots,a_n)=\cone_{\mathbb{Z}}(a_1,\ldots,a_n) + \Pi(a_1,\ldots,a_n).\]

In particular
\begin{eqnarray}
\label{eqn:ehrhart-decomposition}
\mathbb{Z}^d \cap \cone_{\mathbb{R}} (a_1,\ldots,a_n)=\cone_{\mathbb{Z}}(a_1,\ldots,a_n) + (\mathbb{Z}^d \cap \Pi(a_1,\ldots,a_n)).
\end{eqnarray}

This can be phrased in terms of multivariate generating functions. Consider the ring of generating functions in the variables $z_1,\ldots,z_d$ and write $z^x=z_1^{x_1}\cdot\ldots\cdot z_d^{x_d}$ for any integer point $x\in\ZZ^d$. Then
\begin{eqnarray}
\label{eqn:cone-and-parallelepiped}
\sum_{x\in \mathbb{Z}^d\cap\cone_{\mathbb{R}}(a_1,\ldots,a_n)} z^x = \frac{\sum_{x\in\mathbb{Z}^d\cap\Pi(a_1,\ldots,a_n)} z^x}{(1-z^{a_1})\cdot\ldots\cdot(1-z^{a_n})}
\end{eqnarray}
since $\frac{1}{(1-z^{a_1})\cdot\ldots\cdot(1-z^{a_n})}$ is the
multivariate generating function of
$\cone_{\mathbb{Z}^d}(a_1,\ldots,a_n)$. Note that the numerator is
a finite sum, so that if all $a_i$ are non-negative, the numerator is in
fact a polynomial.

Now, let $n<d$ be integers and let $v_1,\ldots,v_n\in\mathbb{Z}^{d-1}_{\geq 0}$ denote the
vertices of an integral simplex $\Delta$ in $\mathbb{R}^{d-1}$. By
embedding $\Delta$ into $\mathbb{R}^d$ at height $x_{d}=1$, we pass to
the vectors $a_1,\ldots,a_n$ with $a_i=(v_{i,1},\ldots,v_{i,{d-1}},1)$
and
\[\#\mathbb{Z}^{d-1}\cap k\Delta = \#\mathbb{Z}^d\cap\cone_{\mathbb{R}}(a_1,\ldots,a_n) \cap \{x\in\mathbb{R}^d \;|\; x_d=k\}\]
which, expressed in terms of generating functions, reads
\begin{eqnarray}
\label{eqn:ehrhart-and-cone}
\sum_{k\geq 0} \sum_{x\in\mathbb{Z}^{d-1}\cap k\Delta} z_1^{x_1}\cdots z^{x_{d-1}}_{d-1}z_d^k = \sum_{x\in\mathbb{Z}^{d}\cap\cone_{\mathbb{R}}(a_1,\ldots,a_n)} z^x.
\end{eqnarray}

Combining identities (\ref{eqn:cone-and-parallelepiped}) and (\ref{eqn:ehrhart-and-cone}), substituting 1 for $z_1,\ldots,z_{d-1}$
and substituting $z$ for $z_d$ we obtain
\[\sum_{k\geq 0} L_{\Delta}(k) z^k = \frac{\sum_{i=0}^{n-1} h_i^* z^i}{(1-z)^n}\]
where $h^*_i$ is the number of lattice points $x\in\Pi(a_1,\ldots,a_n)$
with $x_d=i$.

This completes the proof of Ehrhart's classic interpretation of the
$h^*$-vector.

\begin{theorem}[Ehrhart \cite{Ehrhart1962,EugeneEhrhart1977}] 
Let $v_1,\ldots,v_{n+1}\in\mathbb{Z}^d$
be linearly independent and let $a_i=(v_{i,1},\ldots,v_{i,d},1)$ for
$i=1,\ldots,{n+1}$. Let $h^*=(h^*_0,\ldots,h^*_n)$ denote the
$h^*$-vector of the $n$-dimensional simplex
$\Delta=\conv(v_1,\ldots,v_{n+1})$. Then
\[h^*_i=\#\mathbb{Z}^{d+1}\cap\Pi(a_1,\ldots,a_{n+1})\cap \{x\in\mathbb{R}^{d+1} \;|\; x_{d+1}=i\} .\]
\end{theorem}

By virtue of the fact that polytpoes are convex, the fact that every integral polytope can be
triangulated and using a clever irrational shifting argument to get rid of lattice
points on lower-dimensional faces \cite{Beck2007a}, this theorem can be extended to general lattice polytopes.

\begin{theorem}[Stanley \cite{Stanley1980}] 
Let $K$ be a $d$-dimensional integral
polytope. Then the $h^*$-vector of $K$ is non-negative.
\end{theorem}

Our goal is now to obtain a similar counting interpretation, and, in
particular, a similar non-negativity result for the $f^*$-vector of
polytopal complexes. Before we come to this, we present examples of
polytopal \emph{complexes} where the $h^*$-vector has negative entries.

\subsection{$h^*$-vectors with negative entries}
\label{sec:prelim-col-comp}

Stanley's theorem tells us that in order to find $h^*$-vectors with
negative entries we have to look outside the class of integral
polytopes. We are going to consider integral polytopal complexes
instead.

Coloring complexes of graphs are a class of simplicial complexes that
have been studied by a number of authors in recent years, see, e.g.,  \cite{Breuer2000,Crown2009,Hersh2007,Hultman2007,Jonsson2005,Steingrimsson2001}. All coloring complexes of graphs have a
non-negative $h^*$-vector. A natural generalization are coloring
complexes of hypergraphs. For details about this notion, we refer the
interested reader to \cite{Breuer2000}.

A \emph{hypergraph} $H$ is a finite set $V$ of vertices, together with a
set $E$ of \emph{edges}. An edge is a set of vertices of cardinality at
least two. A proper coloring of $H$ is a labeling $c$ of the vertices of
$H$ with the property that every edge $e\in E$ contains at least two
vertices $i,j\in e$ that have a different color $c_i \not= c_j$. Let $S$
be the set of all vectors in $x\in\{0,1\}^{|V|}$ that are not equal
to the all-one and all-zero vectors. We can now define the simplicial
complex $K$ which is called the \emph{coloring complex} of $H$ as
follows. $\sigma$ is a face of $K$ if and only if 1) $\sigma\subset S$, 2) for
any two vertices $x,y\in\sigma$ we have $x\leq y$ or $y\leq x$
componentwise and 3) there exists an edge $e\in E$ such that for all
vertices $x\in\sigma$ and all $i,j\in e$ we have $x_i=x_j$. Notice that
an element of $x\in S$ appears as a vertex of $K$ if and only if $x$,
viewed as a coloring of the vertices of $H$ with exactly two colors $0$
and $1$, is an improper coloring.

As an example, we consider the hypergraph $H$ on vertex set
$\{1,\ldots,10\}$ with edges $\{1,2,3,4,5,6\}$, $\{4,5,6,7,8,9\}$ and
$\{1,2,3,7,8,9\}$. The associated coloring complex $K$ is 3-dimensional.
It consists of three 3-dimensional spheres $S_1,S_2,S_3$ that share a
single 0-dimensional subsphere $S'$. The spheres $S_i$ are simplicial
complexes which can also be obtained by taking the boundary complex of
the 5-dimensional cube $[0,1]^5$ triangulated by the braid arrangement
and removing the all-zero and all-one vertices (and all incident faces).
Then, the $h^*$-vector of $K$ is
\begin{eqnarray*}
h^*(K,3) & = & h^*(S_1,3) + h^*(S_2,3) + h^*(S_3,3) - 2h^*(S',3) \\
&=& 3\cdot(0,30,60,30) - 2\cdot (2,-6,6,-2) \\
&=& (-4,102,168,94)
\end{eqnarray*}
which has a negative entry.

Intuitively speaking, the reason for the negative entry is that the
complex consists of spheres that have an intersection of codimension
strictly greater than 1. Further examples of hypergraph coloring
complexes with negative entries in their $h^*$-vector can be constructed
in this way.

\section{Parititoning a simplicial cone into discrete cones}
\label{sec:partition}

As we have seen, (\ref{eqn:ehrhart-decomposition}) gives a partition of the set of lattice points in $\cone_{\mathbb{R}}(a_1,\ldots,a_n)$ into discrete cones. This partition is ideally suited for the analysis of the $h^*$-vector. To get our hands onto the $f^*$-vector, however, we need a different partition, given in Theorem~\ref{thm:partition}, which we are going to develop in this section. Theorem~\ref{thm:partition} is the main technical result of this article, as the counting results in subsequent sections can be derived from Theorem~\ref{thm:partition} in a straightforward fashion.

In order to prove this partition result, we first need a couple of
definitions. The basic idea is illustrated in Figure~\ref{fig:cone}.

For every real number $x$ there exist an integer
$\integ(x)\in\mathbb{Z}$ and a real number $\fract(x)\in(0,1]$
such that \[x = \integ(x) + \fract(x).\]

Note that if $x$ is not an integer then
$\integ(x) = \lfloor x \rfloor$ and
$\fract(x) = x - \lfloor x \rfloor$. But if $x\in\mathbb{Z}$, then
$\integ(x) = \lfloor x \rfloor + 1$ and
$\fract(x) = x - \lfloor x \rfloor - 1$. So we call $\integ(x)$
and $\fract(x)$ the \emph{skew integral} and \emph{skew fractional}
part of $x$, respectively. If $v\in\mathbb{R}^d$ is a vector, we use
$\integ(v)$ and $\fract(v)$ to genote the vector of skew
integral and skew fractional parts of the components of $v$,
respectively.

Given linearly independent integer vectors
$v_1,\ldots,v_d\in\mathbb{Z}^d$, we define the \emph{fundamental
simplex} $\Delta(v_1,\ldots,v_d)$ generated by these vectors by
\[\Delta(v_1,\ldots,v_d)= \left\{\sum_{i=1}^d\lambda_i v_i \;\middle|\; 0\leq\lambda_i\in\mathbb{R}, \sum_{i=1}^d\lambda_i \leq d \right\}.\]

The \emph{half-open fundamental simplex} is
\[\Delta^\circ(v_1,\ldots,v_d)=\left\{\sum_{i=1}^d\lambda_i v_i \; \middle|\; 0<\lambda_i\in\mathbb{R}, \sum_{i=1}^d\lambda_i \leq d \right\}.\]

We say a point $z\in\cone_{\mathbb{R}}(v_1,\ldots,v_d)$ is
\emph{at level $k$} if $z=\sum_{i=1}^d\lambda_i v_i$ with
$k-1<\sum_{i=1}^d\lambda_i \leq k$ and define $k=\lev(z)$ to be the \emph{level} of $z$. We denote by
$\Lev(k)$ the set of all lattice points in
$\Delta^\circ(v_1,\ldots,v_d)$ at level $k$.

We now define sets $T_1,\ldots,T_d$ with the property that $T_i\subset \Lev(i)$. The
definition is inductive:
\begin{eqnarray*}
T_1  &=& \Lev(1), \\
T_k &=& \Lev(k) \setminus \left( \bigcup_{i=1}^{k-1} \bigcup_{z\in T_i} z+\cone_{\mathbb{Z}}(v_1,\ldots,v_i)\right).
\end{eqnarray*}

We call the lattice pionts in $\bigcup_{i=1}^\infty T_i$ \emph{atomic}.
If $z=\sum_{i=1}^d\lambda_i v_i =: V\lambda$ for some atomic $z$, then
we also call $\lambda$ atomic. If furthermore $x=\sum_{i=1}^d\mu_i v_i$
then note that for all $k$ we have 
\[ x\in z + \cone_{\mathbb{Z}}(v_1,\ldots,v_k) \;\; \text{ if
and only if } \;\; \mu \in \lambda  + \cone_{\mathbb{Z}}(e_1,\ldots,e_k),\]
where the $e_i$ denote the standard unit vectors.

Similar to our definition of $\lev(z)$, we write $\lev(\lambda)$ to denote the level of $\lambda$, i.e.,
$\lev(\lambda)$ is the unique integer such that
$\lev(\lambda)-1<\sum_{i=1}^d \lambda_i \leq \lev(\lambda)$.

We write $\deg(\lambda)$ to denote the \emph{degree} of $\lambda$:
If there exists an index $1\leq j \leq d$ such that $\lambda_j > 1$,
then $\deg(\lambda)$ is defined to be the smallest such index. If
there is no such index, we let $\deg(\lambda):=d+1$.

So $\lambda\in\mathbb{R}^d_{> 0}$ is atomic if and only if $V\lambda$ is
integer and there does not exist a $\mu\in\mathbb{R}^d_{> 0}$ such that
$V\mu$ is integer, $\mu\not=\lambda$ and
\[\lambda \in \mu + \cone_{\mathbb{Z}}(e_1,\ldots,e_{\lev(\mu)}).\]

These definitions are illustrated in Figure~\ref{fig:cone}. 

\begin{figure}[ht]
\includegraphics[angle=-90,width=9cm]{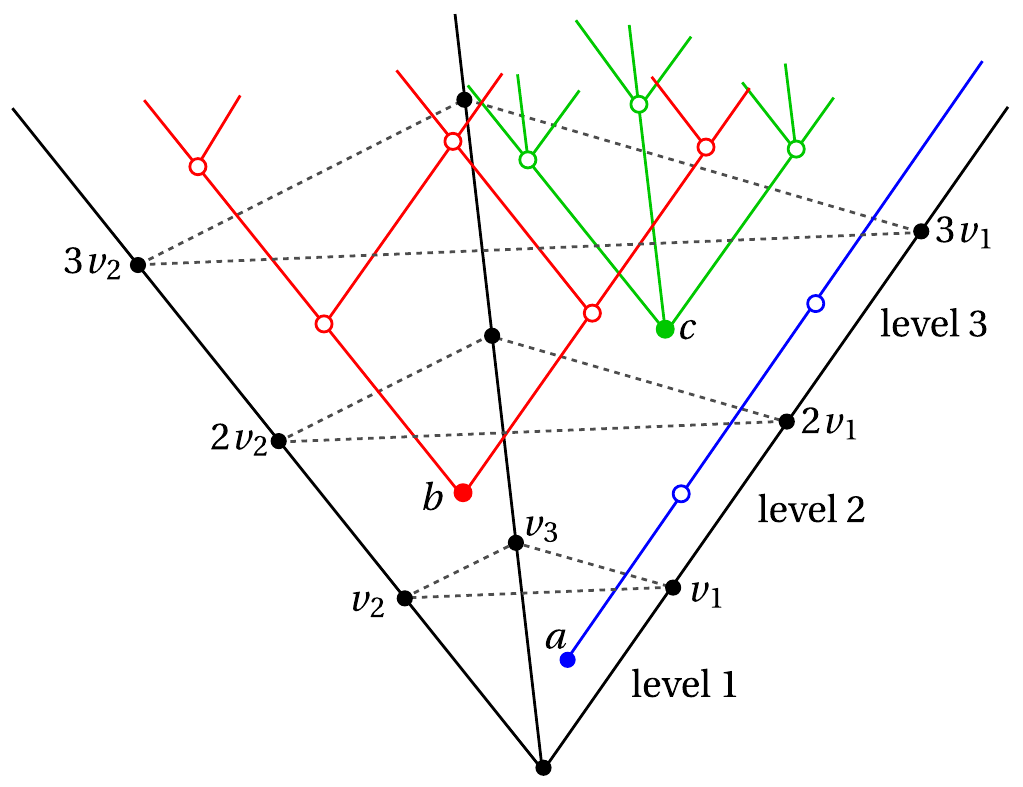}
\caption{\label{fig:cone}This
figure illustrates Theorem~\ref{thm:partition}. (Note that
some lattice points in the cone are not shown here.)
Suppose we have three linearly independent vectors
  $v_1,v_2,v_3\in\RR^3$. The simplex $\conv(3v_1,3v_2,3v_3,0)$ is the
  fundamental simplex of the cone generated by $v_1,v_2,v_3$. The
  half-open fundamental simplex is the intersection of the fundamental simplex with the interior of
  the cone. It is partitioned into three levels as indicated. The first level consists
  of all points below the hyperplane spanned by $v_1,v_2,v_3$,
  including the points on the hyperplane. The second
level contains all points between the hyperplanes spanned by
$v_1,v_2,v_3$ and $2v_1,2v_2,2v_3$, respectively, excluding the points
on the former but including the points on the latter and similarly for
level 3. The set of atomic lattice points in the cone is defined inductively.
In the figure, the solid points in the interior of the cone show
atomic lattice points, whereas the empty circles show lattice points
that are not atomic. All lattice points at level 1 in the open cone are atomic, as
the point labeled $a$ in the figure. Because $a$ is atomic, none of
the points $a+kv_1$ for $1\leq k\in\ZZ$ are atomic. Suppose $b$ is a
lattice point in level 2 that cannot be reached from any atomic
lattice point in level 1 by adding $v_1$. Then $b$ is atomic and all points that can be reached from $b$ by adding $v_1$
or $v_2$ are not atomic. Similarly, for any atomic lattice point $c$ in
level 3, all points that be reached by adding the first three
generators to $c$ are not atomic. There are no atomic lattice points
above level 3. The statement of Theorem~\ref{thm:partition} is that by placing
$i$-dimensional discrete cones at all atomic lattice points in level
$i$ in this way, we obtain a partition of the set of lattice points in
the interior of $\cone_{\mathbb{R}}(v_1,v_2,v_3)$.}
\end{figure}

Despite their inductive definition, it turns out that atomic coefficent
vectors $\lambda$ have a simple characterization.

\begin{lemma}
\label{lem:atomicity} 
Let $z= \sum_{i=1}^n \lambda_i v_i$ be a lattice point
in the interior of $\cone_{\mathbb{R}}(v_1,\ldots,v_d)$.
This means in particular that
$\lambda\in\cone_{\mathbb{R}}(e_1,\ldots,e_d)$. Then:

\begin{enumerate}
\item \label{itm:necessary}
If $\lambda$ is not atomic, then $\deg(\lambda) < \lev(\lambda)$.
\item \label{itm:lev-small}
If $\lev(\lambda) = 1$, then $\deg(\lambda) = d+1$.
\item \label{itm:lev-large}
If $\lev(\lambda)> d$, then $\deg(\lambda) \leq d$.
\item \label{itm:construction}
If $\deg(\lambda) <\lev(\lambda)$, then there exists an atomic $\mu$ such that $\lambda \in \mu+\cone_{\mathbb{Z}}(e_1,\ldots,e_{\lev(\mu)})$.
\item \label{itm:sufficient}
If $\deg(\lambda) < \lev(\lambda)$, then $\lambda$ is not atomic.
\item \label{itm:finite}
 If $\lev(\lambda)>d$, then $\lambda$ is not atomic. In particular, there are only finitely many atomic lattice points.
\end{enumerate}
\end{lemma}

So in particular we have the following \textbf{characterization of atomicity}:
\begin{quote} 
$\lambda$ is atomic if and only if $\lambda_j \leq 1$ for all indices $j < \lev(\lambda)$.
\end{quote}
Or equivalently:
\begin{quote} 
$\lambda$ is atomic if and only if $\deg(\lambda)\geq \lev(\lambda)$.
\end{quote}

\begin{proof}[Proof of Lemma~\ref{lem:atomicity}.] We proceed in several steps.

\medskip\noindent\emph{Part (\ref{itm:necessary}): If $\lambda$ is not atomic, then
$\deg(\lambda)<\lev(\lambda)$.}\smallskip

We have to show that there exists an index $j < \lev(\lambda)$
such that $\lambda_j> 1$. If $\lambda$ is not atomic, then there exists
an atomic $\mu$ with $\lev(\mu) < \lev(\lambda)$ such that
$\lambda\in \mu + \cone_{\mathbb{Z}}(e_1,\ldots,e_{\lev(\mu)})$,
i.e., there exists a non-negative integral vector
$\delta \in \mathbb{Z}^d_{\geq 0}\setminus\{0\}$ such that
$\lambda=\mu +\delta$ with $\delta_i=0$ for all $i>\lev(\mu)$. As
$\lambda \not= \mu$, $\delta_j\geq 1$ for some $j\leq \lev(\mu)$.
Thus $\lambda_j > \delta_j =1$ and
$\deg(\lambda) \leq j\leq \lev(\mu) < \lev(\lambda)$
as desired.

\medskip\noindent\emph{Part (\ref{itm:lev-small}): If $\lev(\lambda) = 1$, then
$\deg(\lambda) = d+1$.}\smallskip

We have $0<\sum_i \lambda_i \leq 1$ and $\lambda_i > 0$ for all $i$. Thus
$\lambda_i \leq 1$ for all $i$.

\medskip\noindent\emph{Part (\ref{itm:lev-large}): If $\lev(\lambda)> d$, then
$\deg(\lambda) \leq d$.}\smallskip

If $\sum_i \lambda_i > d$, then, by the pigeonhole principle, there is
an $i$ such that $\lambda_i > 1$ whence
$\deg(\lambda)\leq i \leq d$.

\medskip\noindent\emph{Part (\ref{itm:construction}): If $\deg(\lambda) <\lev(\lambda)$, then there
exists an atomic $\mu\not= \lambda$ such that
$\lambda \in \mu+\cone_{\mathbb{Z}}(e_1,\ldots,e_{\lev(\mu)})$.}\smallskip

Let $\lev(\lambda)=l$. Let
$\lambda^l,\lambda^{l-1},\ldots,\lambda^k$ be a sequence of coefficient
vectors with $\lev(\lambda^i)=i$ constructed recursively as
follows. We start with $\lambda^l=\lambda$. Given $\lambda^i$, we
distinguish two cases.

\begin{enumerate}[i.]
\item
  If $\deg(\lambda^i) < \lev(\lambda^i)$, then we define the
  next element in our sequence as
  $\lambda^{i-1}=\lambda^{i}- e_{\deg(\lambda^i)}$. In this case,
  $\lev(\lambda^{i-1})=\lev(\lambda^{i})-1=i-1$ and
  $\deg(\lambda^{i-1}) \geq \deg(\lambda^{i}).$ Note that $\deg(\lambda^i)\leq d$ by part (\ref{itm:lev-large}).
\item
  If $\deg({\lambda^i}) \geq \lev(\lambda^i)$, then we stop
  and $\lambda^k=\lambda^i$ is the last element of our sequence. Note
  that $k\geq 1$, as $\lev(\lambda)=1$ implies that
  $\deg(\lambda)=d+1$ by part (\ref{itm:lev-small}).
\end{enumerate}

By part (\ref{itm:necessary}), we know that $\lambda^k$ is atomic as
$\deg(\lambda^k)\geq \lev(\lambda^k)$. By construction, we
know that
\[\deg(\lambda^l) \leq \deg(\lambda^{l-1}) \leq \cdots \leq \deg(\lambda^{k+1}) < \lev(\lambda^{k+1}) = \lev(\lambda^k) + 1\]
whence
\[\lambda^l = \lambda ^k + \sum_{i=k+1}^l e_{\deg(\lambda^i)}\]
where $\deg(\lambda^i)\leq \lev(\lambda^k)$ for all $i=k+1,\ldots,l$ and so
\[\lambda = \lambda^k + \cone_{\mathbb{Z}}(e_1,\ldots,e_{\lev(\lambda^k)})\]
as desired. Note that $\mu:=\lambda^k\not= \lambda$ as
$\deg(\lambda)<\lev(\lambda)$ but
$\deg(\mu)\geq \lev(\mu)$. Note also that $\mu-\lambda$ is
integral, so that $\sum_i \mu_i v_i$ is a lattice point.

\medskip\noindent\emph{Part (\ref{itm:sufficient}): If $\deg(\lambda) < \lev(\lambda)$, then
$\lambda$ is not atomic.}\smallskip

By part (\ref{itm:construction}), it follows that
$\lambda \in \mu + \cone_{\mathbb{Z}}(e_1,\ldots,e_{\lev(\mu)})$
for $\mu\not=\lambda$, which means that $\lambda$ is not atomic.

\medskip\noindent\emph{Part (\ref{itm:finite}): If $\lev(\lambda)>d$, then $\lambda$ is not atomic. In particular, there are only finitely many atomic lattice points.} \smallskip

If $\lev(\lambda)>d$, then $\lev(\lambda)>\deg(\lambda)$ by part (\ref{itm:lev-large}) and so $\lambda$ is not atomic by part (\ref{itm:sufficient}). Since every level contains only finitely many lattice points, it follows that the total number of atomic lattice points is finite.
\end{proof}

After these preparations, we can now show Theorem~\ref{thm:partition},
the partition theorem at the heart of this article.

\begin{proof}[Proof of Theorem~\ref{thm:partition}] 
First, we note that without loss of generality, we can assume $n=d$. Next, we observe that the right-hand side is contained in the left-hand side of (\ref{eqn:partition}) by construction. So we only have to show that the left-hand side is contained in the right-hand side and that the union is disjoint.

\medskip\noindent\emph{The union is disjoint.}\smallskip

Let $\lambda=\alpha+\delta=\beta+\gamma$ where $\alpha$ and $\beta$ are
atomic,
$\delta\in\cone_{\mathbb{Z}}(e_1,\ldots,e_{\lev(\alpha)})$
and
$\gamma\in\cone_{\mathbb{Z}}(e_1,\ldots,e_{\lev(\beta)})$.
Without loss of generality, we assume that
$\lev(\alpha)\leq\lev(\beta)$.

Note that because $\delta$ and $\gamma$ are integer vectors, $\fract(\lambda)=\fract(\alpha)=\fract(\beta)$
and, as both $\alpha$ and $\beta$ are atomic,
$\alpha_i=\fract(\alpha_i)$ for all $i< \lev(\alpha)$ and
$\beta_i=\fract(\beta_i)$ for all $i< \lev(\beta)$, by the
characterization of atomicity. Furthermore,

\begin{itemize}
\item
  $\alpha_i=\beta_i$ for all $i<\lev(\alpha)$ because at these
  indices both $\alpha$ and $\beta$ are fractional, and
\item
  $\alpha_i=\beta_i$ for all $i>\lev(\beta)$ because at these
  indices $\alpha_i=\beta_i=\lambda_i$ by construction.
\end{itemize}
Now we distinguish two cases.

\medskip\noindent\emph{Case 1: $\lev(\alpha)<\lev(\beta)$.}\smallskip

\begin{itemize}
\item
  $\alpha_i\geq \beta_i$ for
  $\lev(\alpha)\leq i < \lev(\beta)$, because $\beta$ is
  fractional at these indices, and
\item
  $\alpha_{i} \geq \beta_i$ for $i=\lev(\beta)$, because
  $\beta_i\leq \lambda_i =\alpha_i$.
\end{itemize}
So $\alpha\geq \beta$ which implies
$\lev(\alpha)\geq \lev(\beta)$, which gives a contradiction.

\medskip\noindent\emph{Case 2: $\lev(\alpha)=\lev(\beta)$.}\smallskip

In this case, we know $\alpha_j =\beta_j$ for all
$j \not= \lev(\alpha)=\lev(\beta)$. So let
$i=\lev(\alpha)$. First we observe that
\[\lambda_i = \alpha_i +\delta_i = \beta_i +\gamma_i\] where $\delta_i$
and $\gamma_i$ are integers, so that $\alpha_i-\beta_i$ is an integer.
Second we argue that because $\lev(\alpha)=\lev(\beta)$,
\[|\alpha_i -\beta_i| = |\sum_j \alpha_j - \sum_j \beta_j| < 1.\] Taking
these observations togther, we obtain $\alpha_i=\beta_i$ and hence
$\alpha=\beta$, as desired.

\medskip\noindent\emph{The left-hand side of (\ref{eqn:partition}) is contained in the right-hand side.}\smallskip

Let $\lambda$ be the coefficient vector of a lattice point in the cone.
If $\deg(\lambda)\geq \lev(\lambda)$, then $\lambda$ is
atomic and hence contained in the right-hand side. Otherwise
$\deg(\lambda)< \lev(\lambda)$ and thus, by part (\ref{itm:construction}) of Lemma~\ref{lem:atomicity}, there exists an atomic $\mu$ such that
$\lambda = \mu + \cone_{\mathbb{Z}}(e_1,\ldots,e_{\lev(\mu)})$,
which shows that $\lambda$ is contained in the right-hand side as well.
\end{proof}

\section{What $f^*$-vectors count}
\label{sec:fstar-counting}

We now apply the partition theorem from the previous section to obtain
results on the $f^*$-vector of polytopes. We start out with the proof
of Theorem~\ref{thm:counting-interpretation}, the counting
interpretation of the $f^*$-coefficients of a lattice simplex. 

\begin{proof}[Proof of Theorem~\ref{thm:counting-interpretation}] Let the vertices of $\Delta$ be denoted by
$v_1,\ldots,v_{d+1}$. Then the vertices of $\Delta\times\{1\}\subset\RR^{n+1}$ are linearly
independent integer vectors $a_1,\ldots,a_{d+1}\in\mathbb{Z}^{n+1}$ with
$a_i=(v_{i,1},\ldots,v_{i,n},1)$. Let
\[C=\cone_{\mathbb{R}}(a_1,\ldots,a_{d+1})\] denote the cone generated
by the $a_i$. As $\Delta$ is open, the number of lattice points in the
$k$-th dilate of $\Delta$ equals the number of lattice points in the
relative interior of $C$ at height $k$, i.e.,
\begin{eqnarray}
\label{eqn:fstar-count-1}
L_\Delta(k)&=&\#\mathbb{Z}^{n+1}\cap \relint(C) \cap \{x\in\mathbb{R}^{n+1} \;|\; x_{n+1} = k\}.
\end{eqnarray}

By Theorem~\ref{thm:partition}, we know
\begin{eqnarray}
\label{eqn:fstar-count-2}
\relint(\cone_{\mathbb{R}}(a_1,\ldots,a_{d+1})) \cap\mathbb{Z}^{n+1} &=& \bigcup_{z\text{ atomic}} z +\cone_{\mathbb{Z}}(a_1,\ldots,a_{\lev(z)})
\end{eqnarray}
where the union is disjoint and runs over all atomic lattice points
$z\in C$. As all $a_i$ have last coordinate equal to 1, we have for any $1\leq l \leq d+1$
\begin{eqnarray*}
\#\mathbb{Z}^{n+1}\cap \cone_{\mathbb{Z}}(a_1,\ldots,a_{l})\cap \{x\in\mathbb{R}^{n+1} \;|\; x_{n+1} = k\} &=& {k+l-1 \choose l-1}.
\end{eqnarray*}

Translating the discrete cone by a lattice point $z$ at level $l$ now
amounts to shifting this polynomial by $l$. Thus, for any atomic lattice
point $z\in C$,
\begin{eqnarray}
\label{eqn:fstar-count-3}
\#\mathbb{Z}^{n+1}\cap \left( z + \cone_{\mathbb{Z}}(a_1,\ldots,a_{\lev(z)}) \right)\cap \{x\in\mathbb{R}^{n+1} \;|\; x_{n+1} = k\} &=& {k-1 \choose \lev(z)-1},
\end{eqnarray}
as any lattice point $x$ in the set on the left-hand side satisfies
$x=z+y$ where
$y\in\cone_{\mathbb{Z}}(a_1,\ldots,a_{\lev(z)})$ and
$z_{n+1}=\lev(z)$. Combining (\ref{eqn:fstar-count-1}), (\ref{eqn:fstar-count-2}) and (\ref{eqn:fstar-count-3}), we obtain
\[L_\Delta(k) = \sum_{i=0}^d \left( \#\text{ atomic $z\in C$ at level $i+1$ }\right) \cdot {k-1 \choose i}\]
which proves the theorem.
\end{proof}

The previous theorems allow us to prove Theorem~\ref{thm:characterization}, a complete characterization of
$f^*$-vectors of integral partial polytopal complexes.

\begin{proof}[Proof of Theorem~\ref{thm:characterization}] 
First, we argue that every integral partial polytopal
complex has a non-negative $f^*$-vector. Let $K$ be an integral
half-open polytopal complex of dimension $d$. Then the support of $K$
can be written as the disjoint union of relatively open lattice
simplices $\sigma_1,\ldots,\sigma_N$, whence $f^*(K,d)$ is the sum of
all $f^*(\sigma_j,d)$. By Theorem~\ref{thm:counting-interpretation}, $f^*(\sigma_j,d)$ is a
non-negative integer for all $j$ and hence, so is $f^*(K,d)$. Note that this
argument only works because
$f^*_i(\sigma_j,\dim(\sigma_j))=f^*_i(\sigma_j,d)$ for all
$0\leq i\leq \dim(\sigma_j)$, a property the $h^*$-vector does not have.

Next, we argue that every non-negative integral vector
$f^*=(f^*_0,\ldots,f^*_d)$ for some $d$ can be realized as the
$f^*$-vector of some integral half-open polytopal complex $K$. This,
however, is straightforward. Given $f^*$ we let $K$ be a polytopal
complex that is the disjoint union of $f^*_i$ open unimodular
$i$-dimensional lattice simplices for each $i=0,\ldots,d$. By
construction $f^*(K,d)=f^*$, as desired.
\end{proof}

This implies in particular that every integral polytopal complex (and
hence every lattice polytope) has a non-negative integral $f^*$-vector
(Theorem~\ref{thm:non-negative}). The crucial point here is that this
holds even if the polytopal complex does not have a unimodular
triangulation and even if its $h^*$-vector does have negative entries.

\section{The rational case}
\label{sec:rational-case}

Ehrhart's theorem for rational polytopes states that if $P$ is a rational polytope, then $L_P(k)$ is a quasipolynomial in $k$. A
\emph{quasipolynomial} is a function $q(k)$ such that there exists a
number $m$ and polynomials $p_0(k),\ldots,p_{m-1}(k)$ such that
$q(km+l)=p_l(k)$ for all integers $k,l$ with $0\leq l < m$. An $m$ with this property is called a \emph{period} of the quasipolynomial $q$, whereas the minimal $m$ with this property is called the \emph{minimal period} of $q$. Note that any positive integer $m$ such that $mP$ is an integral polytope is a period of $L_P$. The \emph{degree} $\deg(q)$ of a quasipolynomial is the maximum degree of the polynomials $p_i$.

It is possible to define $h^*$-vectors for quasipolynomials and thus for rational polytopes and Stanley's non-negativity theorem also applies in this more general case. \cite{Beck2007,Beck2007a} In this section, we show how the above results for the $f^*$-vector can be generalized to the rational case.

For a given quasipolynomial $q$, a given period $m$ of $q$ and an integer $d\geq \deg(q)$ we define the $f^*$-vector of $q$ by
\begin{eqnarray*} 
f^*(q,d,m) &=& (f^*_0(p_0,d),\ldots,f^*_0(p_{m-1},d), \\ && f^*_1(p_0,d),\ldots,f^*_1(p_{m-1},d), \\ && \ldots, \\ && f^*_d(p_0,d),\ldots,f^*_d(p_{m-1},d)).
\end{eqnarray*}

Note that in this case
$$q(km+l) = \sum_{i=0}^d f^*_i(p_l,d) \cdot {k-1 \choose i},$$
for all $0\leq k,l\in\mathbb{Z}$ with $0\leq l< m$ and $(k,l)\not=(0,0)$.

In analogy to the integral case, we define the $f^*$-vector of a rational polytope $P$ (or more generally a rational partial polytopal complex) to be $f^*(L_P,d,m)$ for a given $d\geq \dim(P)$ and a given period $m$ of $L_P$.

Given these conventions, the proof of
Theorem~\ref{thm:fstar-count-rational} now goes as follows.

\begin{proof}[Proof of Theorem~\ref{thm:fstar-count-rational}]
The proof proceeds just as in the integral case, with the following
differences. Let $v_1,\ldots,v_{d+1}$ denote the vertices of $\Delta$
and define generators $a_1,\ldots,a_{d+1}$ by
$a_i=(mv_{i,1},\ldots,mv_{i,n},m)$. Note that the vectors $a_i$ are
integral by definition of $m$, but they are not primitive\footnote{An
  integral vector $z\in\ZZ^n$ is called \emph{primitive} if the line
  segment from $z$ to the origin contains no lattice point except its
  end points. Equivalently, $z$ is primitive if its components have
  greatest common divisor 1.} in
general! We now consider the fundamental simplex of the cone $C=\cone_{\mathbb{R}}(a_1,\ldots,a_{d+1})$ with respect to these generators and observe that the lattice points $z$ in $C$ at level $i+1$ have last coordinate $z_{n+1}=im+l+1$ for some integral $0\leq l < m$. All lattice points $z\in\cone_{\mathbb{Z}}(a_1,\ldots,a_{d+1})$, however, have last coordinate $z_{n+1}=jm$ for some non-negative integer $j$. The theorem then follows. 
\end{proof}

Theorem~\ref{thm:fstar-count-rational} implies the non-negativity
Theorem~\ref{thm:non-negativity-rational} just as in the integral case.

\begin{proof}[Proof of Theorem~\ref{thm:non-negativity-rational}]
Let $K$ be a rational partial polytopal complex of dimension $d=\dim(P)$. Let $m$ be a positive integer such that $mK$ is integral. Let $T$ be a triangulation of $K$ that uses only vertices of $K$. By Theorem~\ref{thm:fstar-count-rational} all open simplices $\sigma$ in $T$ have an Ehrhart quasipolynomial $L_\sigma$ with period $m$. Thus $$f^*(K,d,m) = \sum_{\sigma \in T} f^*(\sigma,d,m)$$ which shows that $f^*(K,d,m)$ is non-negative and integral.
\end{proof}
 
Interestingly, there is another variant of Theorem~\ref{thm:fstar-count-rational} that makes use of a different grading of the cone over $\Delta$.

We define the \emph{restricted partition function} $p_{m_1,\ldots,m_d}(k)$ of positive integers $m_1, \ldots, m_d, k$ to be the coefficient of $z^k$ in the Laurent expansion of the generating function 
\[ \frac{1}{(1-z)^{m_1}\cdot\ldots\cdot(1-z)^{m_d}}, \]
or, equivalently, \[ F_{m_1,\ldots,m_d}(k) = \# \left\{
(\lambda_1,\ldots,\lambda_d) \; \middle|\; 0\leq\lambda_i\in\ZZ, \sum_{i=1}^d
\lambda_i m_i = k \right\}. \]


\begin{proof}[Proof of Theorem~\ref{thm:fstar-count-rational-skew}]
Theorem~\ref{thm:fstar-count-rational-skew} follows from Theorem~\ref{thm:partition} using the observation that for any lattice point $y$ with $0\leq y_{n+1}=s\in\ZZ$ the value $p_{m_1,\ldots,m_{i+1}}(k-s)$ of the restricted partition function equals the number of lattice points $z$ in $y+\cone_{\mathbb{Z}}(a_1,\ldots,a_{i+1})$ with $z_{n+1} = k$.
\end{proof}

\bigskip

\noindent{\textbf{Acknowledgements.}} I would like to thank Matthias Beck for several helpful dicussions and comments on an early version of this paper.

\bibliographystyle{amsplain}
\bibliography{fvector}

\end{document}